\documentclass[11pt,a4paper,oneside,reqno]{amsart}

\usepackage{latexsym}
\usepackage{amsmath}
\usepackage{amsthm}
\usepackage{amssymb}
\usepackage{amsfonts}
\usepackage{fancyhdr}
\usepackage{comment}
\usepackage{color}
\usepackage{mathrsfs}                           
\usepackage{sidenotes}

\theoremstyle{definition}
\newtheorem{thm}{Theorem}[section]

\newtheorem{prop}[thm]{Proposition}

\newtheorem{lemma}[thm]{Lemma}

\numberwithin{equation}{section}

\newcommand{\R}{\mathbb{R}}

\newcommand{\ol}{\overline}

\renewcommand{\l}{\left}
\renewcommand{\r}{\right}

\renewcommand{\det}[1]{~\mbox{det}\left(#1\right)}

\renewcommand{\det}[1]{~\mbox{det}\left(#1\right)}
\newcommand{\eps}{\varepsilon}

\newcommand{\mR}{\mathbb{R}}
\newcommand{\mC}{\mathbb{C}}

\newcommand{\supp}{\mathop{\rm supp}}
\newcommand{\abs}[1]{\lvert #1 \rvert}

\title[Local gauge conditions for ellipticity in conformal geometry]{Local gauge conditions for ellipticity in conformal geometry}

\author{Tony Liimatainen}
\address{Department of Mathematics and Systems Analysis, Aalto University}
\email{tony.liimatainen@aalto.fi}

\author{Mikko Salo}
\address{Department of Mathematics and Statistics, University of Jyv\"askyl\"a}
\email{mikko.j.salo@jyu.fi}

\date{October 14, 2013}

\begin{document}

\begin{abstract}
In this article we introduce local gauge conditions under which many curvature tensors appearing in conformal geometry, such as the Weyl, Cotton, Bach, and Fefferman-Graham obstruction tensors, become elliptic operators. The gauge conditions amount to fixing an $n$-harmonic coordinate system and normalizing the determinant of the metric. We also give corresponding elliptic regularity results and characterizations of local conformal flatness in low regularity settings.
\end{abstract}

\maketitle

\section{Introduction}

In Riemannian geometry, various regularity results for curvature tensors may be obtained via harmonic coordinate systems. We recall the basic idea following \cite{Besse}, \cite{DK}, \cite{Taylor_isometries}.

Let $(M,g)$ be an $n$-dimensional Riemannian manifold, and let $(x^1, \ldots, x^n)$ be a local coordinate system. The Ricci tensor of $g$ has the expression  
$$
R_{ab} = -\frac{1}{2} \Delta g_{ab} + \frac{1}{2} (\partial_a \Gamma_b + \partial_b \Gamma_a) + \text{lower order terms,}
$$
where $\Delta$ is the Laplace-Beltrami operator, $\Gamma^l = g^{ab} \Gamma_{ab}^l$, and $\Gamma_l=g_{la}\Gamma^a$. One also has 
$$
\Gamma^l = -\Delta x^l.
$$
If the coordinate system is harmonic, in the sense that $\Delta x^l = 0$ for all $l$, then the Ricci tensor becomes an elliptic operator. Thus by elliptic regularity, if the Ricci tensor is smooth in harmonic coordinates, then also the metric is smooth in these coordinates. In particular, if the Riemann curvature tensor of a low regularity metric vanishes, this implies that the metric is smooth in harmonic coordinates and thus locally flat by classical arguments.


The Ricci tensor is of course invariant under diffeomorphisms, and these diffeomorphisms correspond to different gauges for the equation $\mathrm{Ric}(g) = h$. The choice of harmonic coordinates may be viewed as a local gauge condition that results in an elliptic equation. In this article we give similar local gauge conditions for the ellipticity of conformal curvature tensors, based on systems of $n$-harmonic coordinates. We say that a function $u$ in $(M,g)$ is $p$-harmonic ($1 < p < \infty$) if 
$$
\delta(\abs{du}^{p-2} du) = 0.
$$
Here $\abs{\,\cdot\,}$ is the $g$-norm and $\delta$ is the codifferential. A local coordinate system is called $p$-harmonic if each coordinate function is $p$-harmonic (the case $p=2$ of course corresponds to the usual harmonic coordinates).

In the earlier article \cite{LiimatainenSalo}, we established the existence of $p$-harmonic coordinate systems for $1 < p < \infty$ on any Riemannian manifold with $C^{r}$, $r>1$, metric tensor. In conformal geometry, the case $p=n$ is special since the class of $n$-harmonic functions is preserved under conformal transformations. This fact was used in \cite{LiimatainenSalo} to study the regularity of such maps: if one fixes $n$-harmonic coordinate systems both in the domain and target manifold, the components of a conformal mapping will solve an elliptic equation. Thus $n$-harmonic coordinates may be considered as a gauge condition for ellipticity in the regularity problem for conformal transformations. In any $n$-harmonic coordinate system, the following \emph{$n$-harmonic gauge condition} is valid \cite{LiimatainenSalo}:
\begin{equation} \label{nharmonic_gauge_condition}
\Gamma^k = -\frac{n-2}{2} \frac{g^{kr} g^{ka} g^{kb}}{g^{kk}} \partial_r g_{ab}.
\end{equation}

We now recall the definition of conformal curvature tensors. If $(M,g)$ is an $n$-dimensional Riemannian manifold, write $R_{abcd}$, $R_{ab}$, and $R$ for the Riemann curvature tensor, Ricci tensor, and scalar curvature, respectively. The Schouten tensor is defined as 
$$
P_{ab} = \frac{1}{n-2} \left( R_{ab} - \frac{R}{2(n-1)} g_{ab} \right).
$$
The Weyl tensor of $(M,g)$ is the $4$-tensor 
$$
W_{abcd} = R_{abcd} + P_{ac} g_{bd} - P_{bc} g_{ad} + P_{bd} g_{ac} - P_{ad} g_{bc},
$$
the Cotton tensor is the $3$-tensor 
$$
C_{abc} = \nabla_a P_{bc} - \nabla_b P_{ac},
$$
and the Bach tensor is the $2$-tensor
$$
B_{ab}=\nabla^k\nabla^lW_{akbl}+\frac{1}{2}R^{kl}W_{akbl}.
$$
We also consider the Fefferman-Graham obstruction tensor 
$$
\mathcal{O}_{ab}=\frac{1}{n-3}\Delta^{n/2-2}\nabla^k\nabla^lW_{akbl}+\text{lower order terms,}
$$
where $\Delta=\nabla^i\nabla_i$ and $n \geq 4$ is even. If $n=4$ then $\mathcal{O}_{ab} = B_{ab}$. These tensors have the following behavior under conformal scaling in various dimensions: $W(cg) = c W(g)$ for $n \geq 4$, $C(cg) = C(g)$ for $n=3$, and $\mathcal{O}(cg) = c^{-\frac{n-2}{2}} \mathcal{O}(g)$ for $n \geq 4$ even. See~\cite{AcheViaclovsky}, \cite{Bach}, \cite{Besse}, \cite{Der}, \cite{FG}, \cite{FGbook}, \cite{Tian} for additional information on these tensors.

If $T$ is one of the above tensors, the next result shows that the equation $T(g) = h$ becomes elliptic under two natural local gauge conditions: one writes both sides of the equation in $n$-harmonic coordinates and considers the conformally normalized metric $\hat{g}_{jk} = \abs{g}^{-1/n} g_{jk}$ with determinant one. Here $\abs{g} = \det{g_{jk}}$, and we use that $n$-harmonic coordinates for $g$ are also $n$-harmonic coordinates for the conformal metric $\hat{g}$.

\begin{thm} \label{thm_main_ellipticity}
Let $(M,g)$ be a smooth $n$-dimensional Riemannian manifold. The Weyl, Cotton, Bach, and Fefferman-Graham obstruction tensors are elliptic operators in $n$-harmonic coordinates, in the sense that after applying both the $n$-harmonic gauge condition \eqref{nharmonic_gauge_condition} and the condition $|\hat{g}|=1$, the linearizations of the resulting operators at $\hat{g}$ are elliptic.
\end{thm}

After interpreting the conformal curvature tensors of low regularity metrics in a suitable sense, elliptic regularity results will follow readily. The regularity conditions will be given in terms of tensors $W_{abc\phantom{d}}^{\phantom{abc}d}$, $C_{abc}$, $\abs{g}^{\frac{n-2}{2n}}\mathcal{O}_{ab}$ which are invariant in conformal scaling $g \mapsto cg$. We use the notation $C_*^r$ for the Zygmund spaces~\cite{T3}.

\begin{thm} \label{thm_main_regularity}
Let $M$ be a smooth $n$-dimensional manifold, and let $g \in C^r_*$ in some system of local coordinates.
\begin{enumerate}
\item[(a)]
If $n \geq 4$, $r > 1$, and $W_{abc\phantom{d}}^{\phantom{abc}d} \in C^s_*$ for some $s > r-2$ in $n$-harmonic coordinates, then $\abs{g}^{-1/n} g_{jk} \in C^{s+2}_*$ in these coordinates.
\item[(b)]
If $n = 3$, $r > 2$, and $C_{abc} \in C^s_*$ for some $s > r-3$ in $n$-harmonic coordinates, then $\abs{g}^{-1/n} g_{jk} \in C^{s+3}_*$ in these coordinates.
\item[(c)]
If $n \geq 4$ is even, $r > n-1$, and $\abs{g}^{\frac{n-2}{2n}}\mathcal{O}_{ab} \in C^s_*$ for some $s > r-n$ in $n$-harmonic coordinates, then $\abs{g}^{-1/n} g_{jk} \in C^{s+n}_*$ in these coordinates.
\end{enumerate}
\end{thm}

Combining the previous theorem with classical results for $C^3$ metrics, one obtains conditions for local conformal flatness of low regularity metrics via the vanishing of Weyl and Cotton tensors.

\begin{thm} \label{thm_main_conformalflatness}
Let $M$ be a smooth $n$-dimensional manifold, and let $g\in C^r_*$ in some system of local coordinates defined near $p \in M$.
\begin{enumerate}
\item[(a)]
If $n \geq 4$, $r > 1$, and if $W_{abcd} = 0$ near $p$, then $g_{jk} = c \delta_{jk}$ for some positive function $c \in C^r_*$ in some $n$-harmonic coordinates near $p$.
\item[(b)]
If $n = 3$, $r > 2$, and if $C_{abc} = 0$ near $p$, then $g_{jk} = c \delta_{jk}$ for some positive function $c \in C^r_*$ in some $n$-harmonic coordinates near $p$.
\item[(c)]
If $n \geq 4$ is even, $r > n-1$, and if $\mathcal{O}_{ab} = 0$ near $p$, then $\abs{g}^{-1/n} g_{jk}$ is $C^{\infty}$ in any system of $n$-harmonic coordinates near $p$.
\end{enumerate}
\end{thm}

There are a number of results in conformal geometry that employ various gauges to obtain elliptic or parabolic regularity and existence results. The works \cite{AcheViaclovsky}, \cite{Anderson}, \cite{Helliwell}, \cite{Tian} involve constant scalar curvature gauges, sometimes obtained via the solution of the Yamabe problem \cite{LeeParker}, and study asymptotically locally Euclidean manifolds with obstruction flat metrics or boundary regularity of conformally compact Einstein metrics. In~\cite{BahuaudHelliwell} a version of the DeTurck trick is used to show that certain analogues of the Ricci flow, including a modified flow involving the obstruction tensor, are locally well posed. In~\cite{GurskyViaclovsky1}, \cite{GurskyViaclovsky2} quadratic curvature functionals are studied with an emphasis on the structure of the space of their critical metrics. In these papers the corresponding Euler-Lagrange equations become elliptic via suitable gauge conditions.

Let us describe an argument from \cite{Anderson}, \cite{Helliwell}, \cite{Tian} for Bach flat metrics with $n=4$. Choosing a conformal metric with constant scalar curvature, the equation of Bach flatness for the conformally scaled metric reads
$$
\nabla^k \nabla_k R_{ab} + \text{lower order terms} = 0.
$$
This becomes a fourth order elliptic equation in harmonic coordinates for the conformally scaled metric. This fact could be used to give another proof of an analogue of Theorem~\ref{thm_main_regularity}, if a solution to the (local) Yamabe problem for sufficiently low regularity metrics is provided.

We thus have two possible sets of gauge conditions for ellipticity of conformal curvature tensors: ($n$-harmonic coordinates + determinant one) and (constant scalar curvature + harmonic coordinates). Let us compare these briefly. The first approach works for $C^{1,\alpha}$ metrics with $\alpha > 0$ and relies on an $n$-harmonic coordinate system. In \cite{LiimatainenSalo} such coordinates $(u^1,\ldots,u^n)$ near $p$ were constructed by choosing some smooth coordinates $(x^1,\ldots,x^n)$ near $p$ and then solving for each $j$ the Dirichlet problem 
$$
\delta(\abs{du^j}^{n-2} du^j) = 0 \text{ in } B(p,\eps), \quad u^j|_{\partial B(p,\eps)} = x^j.
$$
An easy variational argument shows that this problem has a unique solution. The point is to show that $(u^1,\ldots,u^n)$ is a $C^1$ diffeomorphism if $\eps$ is small enough, and this relies on standard regularity theory for quasilinear equations (essentially the fact that any $p$-harmonic function is $C^{1,\alpha}$ regular). Given $n$-harmonic coordinates, the conformal normalization (determinant one) is trivial.

The second approach based on a constant scalar curvature gauge requires solving the Yamabe problem (a semilinear equation with a constraint) at least in a small neighborhood; the global Yamabe problem is solvable for $C^{1,1}$ metrics (see \cite[Section VII.7]{ChoquetBruhat}), and we refer to \cite{Escobar}, \cite{HG}, \cite{Ma} for versions of the local problem. The recent paper \cite{HG} deals with $W^{1,p}$, $ p > n$, metrics. Given a conformal metric with constant scalar curvature, the other gauge condition involves a standard harmonic coordinate system.

A possible benefit of the $n$-harmonic coordinate approach is the similarity to harmonic coordinates in Riemannian geometry. In fact, it appears that $n$-harmonic coordinates are a natural generalization of harmonic coordinates in conformal geometry, and other Riemannian applications of harmonic coordinates might have conformal analogues. On the other hand, $n$-harmonic coordinates exist only locally, and if global gauge conditions are required then scalar curvature gauges are available via the Yamabe problem. Finally, $n$-harmonic coordinates exist at least for $C^{1,\alpha}$ metrics whereas global results for the Yamabe problem in the literature seem to require metrics with two derivatives. One motivation for this work was to find characterizations of local conformal flatness beyond the classical $C^3$ case, and it is an interesting question (also asked in \cite[Section 2.7]{IwaniecMartin}) if one can have such characterizations for very low regularity metrics.

\vspace{12pt}
\noindent {\bf Acknowledgements.}
T.L.~was supported in part by the Finnish National Graduate School on Mathematics and its Applications, and M.S.~is partly supported by the Academy of Finland and an ERC Starting Grant. The second author is grateful to Spyros Alexakis, Michael Eastwood, Robin Graham and Gantumur Tsogtgerel for helpful discussions, and in particular to Robin Graham for pointing out several useful references.

\section{Symbol computations} \label{section_symbolcomputations}

In this section we assume that the metric is smooth ($=C^{\infty}$) in some system of local coordinates. We start by recalling standard local coordinate formulas. The Christoffel symbols are given by 
$$
\Gamma_{ab}^k = \frac{1}{2} g^{kl}(\partial_a g_{bl} + \partial_b g_{al} - \partial_l g_{ab}).
$$
Define 
$$
\Gamma^k = g^{ab} \Gamma_{ab}^k, \qquad \Gamma_k = g_{kl} \Gamma^l.
$$
Noting the identity $g^{ab} \partial_l g_{ab} = \partial_l(\log\abs{g})$, we see that 
\begin{equation*}
\Gamma^k = -\partial_l g^{kl} - \frac{1}{2} g^{kl} \partial_l(\log\abs{g}), \qquad 
\Gamma_j = g^{kl} \partial_l g_{jk} - \frac{1}{2} \partial_j(\log\abs{g}). 
\end{equation*}
This also implies that 
$$
\Gamma_{ab}^a = \frac{1}{2} \partial_b(\log\abs{g}).
$$
The Riemann curvature tensor is given by 
$$
R_{abcd} = \langle (\nabla_a \nabla_b - \nabla_b \nabla_a) \partial_c, \partial_d \rangle.
$$
Here $\nabla_a \nabla_b \partial_c = \nabla_a (\Gamma_{bc}^m \partial_m) = \partial_a \Gamma_{bc}^m \partial_m + \Gamma_{bc}^m \Gamma_{am}^r \partial_r$, so that 
\begin{align*}
R_{abcd} = \partial_a \Gamma_{bc}^m g_{md} + \Gamma_{bc}^m \Gamma_{am}^r g_{rd} - \partial_b \Gamma_{ac}^m g_{md} - \Gamma_{ac}^m \Gamma_{bm}^r g_{rd}.
\end{align*}
The Ricci tensor, $R_{bc} = R_{abc}^{\phantom{abc}a}$, is given by 
\begin{align*}
R_{bc} &= \partial_a \Gamma_{bc}^a - \partial_b \Gamma_{ac}^a + \Gamma_{bc}^m \Gamma_{am}^a - \Gamma_{ac}^m \Gamma_{bm}^a \\
 &= \partial_a \Gamma_{bc}^a - \Gamma_{ac}^m \Gamma_{bm}^a - \frac{1}{2} \partial_{bc}(\log\abs{g}) + \frac{1}{2} \Gamma_{bc}^m \partial_m(\log\abs{g}).
\end{align*}
The scalar curvature is $R = g^{bc} R_{bc}$. The Schouten, Weyl, Cotton and Bach tensors were defined in the introduction. By using Bianchi identities as in \cite{Tian}, the Cotton and Bach tensors can also be written as 
\begin{gather*}
(3-n) C_{abc} = \nabla^l W_{abcl}, \\
B_{ab} = (3-n)(- \nabla^k \nabla_k P_{ab} + \nabla^k \nabla_a P_{bk}) + \frac{1}{2} R^{kl} W_{akbl}.
\end{gather*}

Below, we write $T_k$ for quantities that depend smoothly on components of the metric and their derivatives up to order $k$. We will also write 
$$
L = g^{ab} \partial_{ab}, \qquad L^2 = g^{ab} g^{cd} \partial_{abcd}
$$
for the principal parts of the Laplace-Beltrami operator and its square.

\subsection{Bach tensor}
All ellipticity results in this paper ultimately reduce to the fact that the Bach tensor is elliptic in $n$-harmonic coordinates when acting on metrics with determinant one. In the following, if $T=T(g)$ is a nonlinear differential operator acting on metrics, we denote by $\sigma(T)$ the principal symbol of its linearization at $g$.

\begin{lemma} \label{lemma_bach_computation}
The components of the Bach tensor satisfy 
\begin{align*}
\frac{2 (n-2)}{3-n} B_{ab} = L^2 &(g_{ab}) - L (\partial_a \Gamma_b + \partial_b \Gamma_a) \\
 &+ \frac{n-2}{n-1} \partial_{abl} \Gamma^l + \frac{1}{n-1} L(\partial_l \Gamma^l) g_{ab} + T_3
\end{align*}
in any local coordinate system in which $\abs{g}=1$.
\end{lemma}
\begin{proof}
Assume that the metric satisfies $\abs{g}=1$ in a given system of local coordinates.  Then 
$$
R_{bc} = \partial_a \Gamma_{bc}^a + T_1, \qquad R = \partial_a \Gamma^a + T_1.
$$
The Schouten tensor satisfies 
$$
(n-2) P_{ab} = \partial_l \Gamma_{ab}^l - \frac{1}{2(n-1)} (\partial_l \Gamma^l) g_{ab} + T_1.
$$
Recall that $\nabla_m F_{i_1 \ldots i_k} = (\nabla F)_{i_1 \ldots i_k m} = \partial_m F_{i_1 \ldots i_k} - \sum_{s=1}^k F_{i_1 \ldots p \ldots i_k} \Gamma_{m i_s}^p$. Therefore 
$$
(n-2) \nabla_c P_{ab} = \partial_{cl} \Gamma_{ab}^l - \frac{1}{2(n-1)} (\partial_{cl} \Gamma^l) g_{ab} + T_2
$$
and 
$$
(n-2) \nabla_d \nabla_c P_{ab} = \partial_{cdl} \Gamma_{ab}^l - \frac{1}{2(n-1)} (\partial_{cdl} \Gamma^l) g_{ab} + T_3.
$$
For the Bach tensor, we obtain 
\begin{align*}
 &\frac{n-2}{3-n} B_{ab} = (n-2) \left[ g^{km} \nabla_k \nabla_a P_{bm} - g^{km} \nabla_k \nabla_m P_{ab} \right] + T_2 \\
 &= g^{km} \left[ \partial_{akl} \Gamma_{bm}^l - \partial_{klm} \Gamma_{ab}^l + \frac{1}{2(n-1)} \left[ (\partial_{klm} \Gamma^l) g_{ab} - (\partial_{akl} \Gamma^l) g_{bm} \right] \right] + T_3.
\end{align*}
Next we observe that 
\begin{align*}
\partial_{klm} \Gamma_{ab}^l &= \frac{1}{2} \partial_{klm} \left[ g^{lr}(\partial_a g_{br} + \partial_b g_{ar} - \partial_r g_{ab}) \right] \\
 &= \frac{1}{2} g^{lr} (\partial_{aklm} g_{br} + \partial_{bklm} g_{ar} - \partial_{rklm} g_{ab}) + T_3.
\end{align*}
This implies 
\begin{align*}
 &2g^{km} (\partial_{akl} \Gamma_{bm}^l - \partial_{klm} \Gamma_{ab}^l) \\
 &= g^{km} g^{lr} \left[ \partial_{aklm} g_{br} + \partial_{bkla} g_{mr} - \partial_{rkla} g_{mb} \r.\\
 &\ \ \  \l. -\partial_{aklm} g_{br} - \partial_{bklm} g_{ar} + \partial_{rklm} g_{ab} \right] + T_3 \\
 &= L^2 (g_{ab}) - L(g^{km} \partial_{ka} g_{mb}) - L(g^{lr} \partial_{bl} g_{ar}) + g^{km} g^{lr} \partial_{bkla} g_{mr} + T_3 \\
 &= L^2 (g_{ab}) - L(\partial_a \Gamma_b + \partial_b \Gamma_a) + \partial_{abl} \Gamma^l + T_3.
\end{align*}
In the last step we used that $\Gamma_b = g^{km} \partial_k g_{mb}$ when $\abs{g}=1$, and $\Gamma^l = g^{lr} \Gamma_r$. For the remaining terms in the Bach tensor, we use that 
$$
g^{km} \left[ (\partial_{klm} \Gamma^l) g_{ab} - (\partial_{akl} \Gamma^l) g_{bm} \right] = L(\partial_l \Gamma^l) g_{ab} - \partial_{abl} \Gamma^l + T_3.
$$
Collecting these facts, we have proved the lemma.
\end{proof}

Notice that in any harmonic coordinate system in which $\abs{g}=1$, the Bach tensor would have the form 
$$
2(n-2) B_{ab} = (3-n)L^2(g_{ab}) + T_3.
$$
The operator on the right is obviously elliptic. However, in general it is not clear how to find harmonic coordinates with $\abs{g}=1$ even after conformal scaling, owing to the fact that the equations $\Gamma^k = 0$ for harmonic coordinates are not conformally invariant. This is where $n$-harmonic coordinates become useful: they allow the conformal normalization $\abs{g}=1$, and the Bach tensor turns out to be elliptic in these coordinates.

\begin{lemma} \label{lemma_bach_nharmonic}
The components of the Bach tensor satisfy 
\begin{multline*}
\frac{2 (n-2)}{3-n} B_{ab} = L^2 (g_{ab}) - L (\partial_a \tilde{\Gamma}_b + \partial_b \tilde{\Gamma}_a) \\
+ \frac{n-2}{n-1} \partial_{abl} \tilde{\Gamma}^l + \frac{1}{n-1} L(\partial_l \tilde{\Gamma}^l) g_{ab} + T_3
\end{multline*}
in any $n$-harmonic coordinate system with $\abs{g}=1$, where $\tilde{\Gamma}^k$ is as in \eqref{nharmonic_gauge_condition}:
$$
\tilde{\Gamma}^k = -\frac{n-2}{2} \frac{g^{kr} g^{ka} g^{kb}}{g^{kk}} \partial_r g_{ab}.
$$
\end{lemma}
\begin{proof}
Follows from Lemma \ref{lemma_bach_computation} and the fact that $\Gamma^k = \tilde{\Gamma}^k$ in $n$-harmonic coordinates (see \cite[Remark after Theorem 2.1]{LiimatainenSalo}).
\end{proof}

Consider the right hand side operator in Lemma \ref{lemma_bach_nharmonic} acting on symmetric positive definite matrix functions, and denote by $Q = Q_g$ the principal part of the linearization of this operator at $g$. Then $Q$ is the principal part of the following linear operator, acting on symmetric matrix functions $h$ by 
\begin{multline*}
h_{ab} \mapsto L^2 (h_{ab}) - L \left[ \partial_a (g_{bl} \tilde{\Gamma}^l(h)) + \partial_b (g_{al} \tilde{\Gamma}^l(h)) \right] \\
+ \frac{n-2}{n-1} \partial_{abl} \tilde{\Gamma}^l(h) + \frac{1}{n-1} L(\partial_l \tilde{\Gamma}^l(h)) g_{ab}
\end{multline*}
where $L = L_g = g^{ab} \partial_{ab}$, and $\tilde{\Gamma}^k = \tilde{\Gamma}^k_g$ is the operator 
$$
\tilde{\Gamma}^k(h) = -\frac{n-2}{2} \frac{g^{kr} g^{ka} g^{kb}}{g^{kk}} \partial_r h_{ab}.
$$

\begin{lemma}\label{ellipt_calc}
$Q$ is elliptic of order $4$ in dimensions $n \geq 3$.
\end{lemma}
\begin{proof}
The principal symbol of $Q$, acting on a symmetric matrix function $h$, is given by 
\begin{multline*}
(q(\xi) h)_{ab} = \abs{\xi}^4 h_{ab} - \abs{\xi}^2 \left[ \xi_a g_{bl} + \xi_b g_{al} \right] \sigma(-i\tilde{\Gamma}^l)(h) \\
 + \frac{n-2}{n-1} \xi_a \xi_b \xi_l \sigma(-i\tilde{\Gamma}^l)(h)+ \frac{1}{n-1} \abs{\xi}^2 \xi_l \sigma(-i\tilde{\Gamma}^l)(h) g_{ab}
\end{multline*}
where 
$$
\sigma(-i\tilde{\Gamma}^k)(h) = -\frac{n-2}{2} \xi^k \frac{h^{kk}}{g^{kk}}.
$$
Here, the norms $\abs{\,\cdot\,}$ and the raising and lowering of indices are taken with respect to the metric $g$. We also have 
\begin{multline*}
(q(\xi) h)^{ab} = \abs{\xi}^4 h^{ab} - \abs{\xi}^2 \left[ \xi^a \sigma(-i\tilde{\Gamma}^b)(h) + \xi^b \sigma(-i\tilde{\Gamma}^a)(h) \right]  \\
 + \frac{n-2}{n-1} \xi^a \xi^b \xi_l \sigma(-i\tilde{\Gamma}^l)(h)+ \frac{1}{n-1} \abs{\xi}^2 \xi_l \sigma(-i\tilde{\Gamma}^l)(h) g^{ab}.
\end{multline*}
Taking $a=b$, this gives 
\begin{multline*}
(q(\xi) h)^{aa} = \abs{\xi}^4 h^{aa} - 2 \abs{\xi}^2 \xi^a \sigma(-i\tilde{\Gamma}^a)(h)   \\
 + \frac{n-2}{n-1} (\xi^a)^2 \xi_l \sigma(-i\tilde{\Gamma}^l)(h)+ \frac{1}{n-1} \abs{\xi}^2 \xi_l \sigma(-i\tilde{\Gamma}^l)(h) g^{aa} \\
 = \left( g^{aa} \abs{\xi}^2 + (n-2) (\xi^a)^2 \right) \left[ \abs{\xi}^2 \frac{h^{aa}}{g^{aa}} - \frac{n-2}{2(n-1)} \sum_{l=1}^n \xi_l \xi^l \frac{h^{ll}}{g^{ll}} \right].
\end{multline*}
If $\xi \neq 0$ and if $q(\xi)h = 0$, then also 
$$
\frac{h^{aa}}{g^{aa}} - \frac{n-2}{2(n-1)} \sum_{l=1}^n \frac{\xi_l \xi^l}{\abs{\xi}^2} \frac{h^{ll}}{g^{ll}} = 0
$$
for all $a = 1, \ldots, n$. This implies that 
$$
\frac{h^{11}}{g^{11}} = \ldots = \frac{h^{nn}}{g^{nn}} = \lambda
$$
for some function $\lambda$, and therefore  
$$
\left( 1 - \frac{n-2}{2(n-1)} \right) \lambda = 0.
$$
Thus $h^{aa} = 0$ for all $a$, so $\sigma(-i\tilde{\Gamma}^a)(h) = 0$ for all $a$ and $(q(\xi) h)^{ab} = \abs{\xi}^4 h^{ab} = 0$ for all $a, b$. It follows that $h = 0$, showing that $Q$ is elliptic.
\end{proof}

We are now ready to prove the main ellipticity result.

\begin{proof}[Proof of Theorem \ref{thm_main_ellipticity}]
Let $(M,g)$ be a smooth Riemannian manifold, fix an $n$-harmonic coordinate system, and let $\hat{g}_{jk} = \abs{g}^{-1/n} g_{jk}$ in these coordinates. Using the conformal invariance of the $n$-harmonic equation, our coordinate system is $n$-harmonic also with respect to the conformal metric $\hat{g}$. Let $\hat{B}_{ab}$ be the Bach tensor for $\hat{g}$ in $n$-harmonic coordinates after applying the $n$-harmonic gauge condition \eqref{nharmonic_gauge_condition} and the condition $\abs{\hat{g}} = 1$. It follows from Lemmas \ref{lemma_bach_computation}--\ref{ellipt_calc} that the principal part of the linearization of $\hat{B}_{ab}$ is elliptic in dimensions $n \geq 4$.

Moving on to the Weyl tensor, the formula $B_{ab}=\nabla^k\nabla^l W_{akbl}+\frac{1}{2}R^{kl}W_{akbl}$ implies that 
$$
\xi^k \xi^l \sigma(\hat{W}_{akbl})(\xi)h = -\sigma(\hat{B}_{ab})(\xi)h
$$
where $\hat{W}$ is the Weyl tensor of $\hat{g}$ after applying the gauge conditions, and $\sigma$ denotes the principal symbol of the linearization applied to a matrix function $h$. Thus $\sigma(\hat{W}_{akbl})(\xi)h = 0$ implies $h = 0$ by the ellipticity of $\hat{B}_{ab}$, showing that $\hat{W}_{abcd}$ is overdetermined elliptic (its principal symbol is injective). The ellipticity of Cotton and obstruction tensors follow in a similar way from the identities $(3-n) C_{abc} = \nabla^l W_{abcl}$, $(n-3) \mathcal{O}_{ab}=\Delta^{n/2-2}\nabla^k\nabla^lW_{akbl}+T_{n-1}$.
\end{proof}

\section{Elliptic regularity results}


We now consider low regularity Riemannian metrics, and establish elliptic regularity results for conformal curvature tensors. The Weyl tensor will be considered in detail and the arguments for the other tensors will be sketched. We first assume that 
$$
g_{jk} \in W^{1,2} \cap L^{\infty}
$$
in some system of local coordinates near a point. (In this section all function spaces are assumed to be of the local variety near a point, that is, we write $W^{1,2}$ instead of $W^{1,2}_{\mathrm{loc}}(U)$ etc.) This seems to be a minimal assumption for defining the Weyl tensor: the set $W^{1,2} \cap L^{\infty}$ is an algebra under pointwise multiplication \cite{KatoPonce}, \cite{BadrBernicotRuss}, and therefore $g^{jk}, \abs{g} \in W^{1,2} \cap L^{\infty}$.  We see that 
\begin{gather*}
\Gamma_{ab}^c \in L^2, \\
R_{abc}^{\phantom{abc}d} = \partial_a \Gamma_{bc}^d - \partial_b \Gamma_{ac}^d + \Gamma_{bc}^m \Gamma_{am}^d - \Gamma_{ac}^m \Gamma_{bm}^d \in W^{-1,2} + L^1, \\
R_{bc} = R_{abc}^{\phantom{abc}a} \in W^{-1,2} + L^1.
\end{gather*}
Since 
$$
P_{ab} = \frac{1}{n-2} \left[ R_{ab} - \frac{1}{2(n-1)} g^{rs} R_{rs} g_{ab} \right],
$$
we have 
\begin{multline*}
W_{abc}^{\phantom{abc}d} = R_{abc}^{\phantom{abc}d} + P_{ac} \delta_b^{\phantom{b}d} - P_{bc} \delta_a^{\phantom{a}d} + P_{bm} g^{md} g_{ac} - P_{am} g^{md} g_{bc} \\
 = R_{abc}^{\phantom{abc}d} + \frac{1}{n-2} \left[ R_{ac} \delta_b^{\phantom{b}d} - R_{bc} \delta_a^{\phantom{a}d} + R_{bm} g^{md} g_{ac} - R_{am} g^{md} g_{bc} \right] \\
  + \frac{1}{(n-2)(n-1)} \left[ g^{rs} g_{bc} R_{rs} \delta_a^{\phantom{a}d} - g^{rs} g_{ac} R_{rs} \delta_b^{\phantom{b}d} \right].
\end{multline*}
Now, since multiplication by an $W^{1,2} \cap L^{\infty}$ function maps $W^{1,q}$ to $W^{1,2}$ for any $q > n$, it also maps $W^{-1,2}$ to $W^{-1,q'}$, and we have 
\begin{gather*}
g^{ij} g_{kl} R_{rs} \in W^{-1,q'} + L^1.
\end{gather*}
This shows that if $g_{jk} \in W^{1,2} \cap L^{\infty}$ in some system of local coordinates, then one can make sense of the components $W_{abc}^{\phantom{abc}d}$ of the Weyl tensor as elements in $W^{-1,q'} + L^1$ in these coordinates for any $q > n$. It is easy to see that also $R_{abcd}, W_{abcd} \in W^{-1,q'} + L^1$ in this case, and the identity $W(cg) = cW(g)$ remains true if $c, g_{jk} \in W^{1,2} \cap L^{\infty}$.

By similar arguments, if $g_{jk} \in W^{2,2} \cap W^{1,\infty}$ in some local coordinates, the components $C_{abc}$ of the Cotton tensor may be interpreted as elements of $W^{-1,2}$ and one has $C(cg) = C(g)$ if $c \in W^{2,2} \cap W^{1,\infty}$ and $n=3$. Moreover, if $g_{jk} \in W^{n-1,2} \cap W^{n-2,\infty}$, then the Fefferman-Graham obstruction tensor satisfies $\mathcal{O}_{ab} \in W^{-1,2}$ and $\mathcal{O}(cg) = c^{\frac{n-2}{2}} \mathcal{O}(g)$ for $n \geq 4$ even. We omit the details of these computations.

We will need the following elliptic regularity result, which is, after using suitable cutoffs and extensions, a consequence of \cite[Theorem 14.4.2]{T3} (the argument also applies to overdetermined elliptic operators). For completeness, the details are given in Appendix \ref{app}.

\begin{prop} \label{elliptic_regularity_differential_operator}
Let $B \subset \mR^n$ be a ball, and let $p(x,D) = \sum_{\abs{\alpha} \leq m} p_{\alpha}(x) D^{\alpha}$ be an $M \times N$ matrix differential operator with coefficients in $C^r_*(B)$ where $r > 0$. Assume that $p(x,D)$ is overdetermined elliptic in the sense that its principal symbol $p_m(x,\xi)$ is injective for $x \in B$ and $\xi \neq 0$. If in the sense of distributions 
$$
p(x,D) u = f \quad \text{in } B
$$
where $u \in C^{m-r+\eps}_*(B)$ for some $\eps > 0$ and $f \in C^s_*(B)$ with $-r < s < r$, then $u \in C^{m+s}_*(B)$.
\end{prop}

The next result gives a symbol computation for the Weyl tensor with low regularity metric.

\begin{lemma} \label{weyl_princip}
Let $g_{jk} \in C^r_*$, $r > 1$, in some system of local coordinates, and assume that $\abs{g}=1$. The principal symbol $\sigma(W_{abcd})$ of the linearization of the Weyl tensor at $g$ satisfies 
 \begin{multline}\label{weyl_contr}
  \xi^a \xi^d \sigma(W_{abcd})h=\frac{n-3}{2(n-2)} \Big[ |\xi|^4 h_{bc} - |\xi|^2 \left( \xi_b\sigma(-i\Gamma_c)h + \xi_c\sigma(-i\Gamma_b)h \right)  \\
  +\frac{n-2}{n-1} \xi_b \xi_c \xi_l \sigma(-i\Gamma^l)h + \frac{1}{n-1} |\xi|^2 \left( \xi_l \sigma(-i\Gamma^l)h \right) g_{bc} \Big].
\end{multline}
\end{lemma}
\begin{proof}
If $g$ were sufficiently smooth, for instance $C^4$, we could readily use the formula $B_{ab}=\nabla^k\nabla^lW_{akbl}+\frac{1}{2}R^{kl}W_{akbl}$ to have
$$
\xi^a \xi^d \sigma(W_{abcd})h=-\sigma(B_{bc})h,
$$
which together with Lemma~\ref{lemma_bach_computation} would yield the claim. For $C^r_*$ metrics, $r > 1$, we have to be more careful because the derivation of the different formulas for the Bach tensor uses the Bianchi identities
$$
\nabla^d R_{abcd} = \nabla_aR_{bc}-\nabla_bR_{ac}, \qquad \nabla^b R_{ab} = \frac{1}{2}\nabla_b R,
$$
that require more derivatives. We can however emulate the Bianchi identities on the symbol level using the equations
\begin{align}
 \xi^d &\sigma(R_{abcd}+R_{ac}g_{bd}-R_{bc}g_{ad})=0, \label{Bianchi_emulated1} \\
 \xi^a &\sigma(R_{ab}g_{cd}-\frac{1}{2}Rg_{ab}g_{cd})=0. \label{Bianchi_emulated2}
\end{align}
By the discussion earlier in this section, one can check that these identities are valid for $C^r_*$ metrics when $r > 1$. We now write the Weyl tensor as
\begin{multline*}
W_{abcd}=(R_{abcd}+R_{ac}g_{bd}-R_{bc}g_{ad})+\l(\frac{1}{n-2}-1\r)(R_{ac}g_{bd}-R_{bc}g_{ad}) \\
  +\frac{1}{n-2}(R_{bd}g_{ac} -R_{ad}g_{bc})-\frac{R}{(n-1)(n-2)}(g_{ac}g_{bd}-g_{bc}g_{ad}) \\
  =(R_{abcd}+R_{ac}g_{bd}-R_{bc}g_{ad})-\frac{n-3}{n-2}(R_{ac}g_{bd}-\frac{1}{2}Rg_{ac}g_{bd}) \\
  +\frac{1}{n-2}(R_{bd}g_{ac} -\frac{1}{2}Rg_{bd}g_{ac}-R_{ad}g_{bc}+\frac{1}{2}Rg_{ad}g_{bc}) \\
  +\frac{n-3}{n-2}R_{bc}g_{ad}-\frac{n-3}{2(n-2)}Rg_{ac}g_{bd} + \frac{1}{2(n-2)}(Rg_{bd}g_{ac}-Rg_{ad}g_{bc}). \\
  -\frac{R}{(n-1)(n-2)}(g_{ac}g_{bd}-g_{bc}g_{ad})
\end{multline*}
Now, taking the symbol of the Weyl tensor in this form and contracting by $\xi^a\xi^d$ we have by the identities~\eqref{Bianchi_emulated1}, \eqref{Bianchi_emulated2} that
\begin{multline}\label{Wc}
 \xi^a \xi^d \sigma(W_{abcd})h=-\frac{n-3}{2(n-2)}\Big[-2\abs{\xi}^2\sigma(R_{bc})h+\frac{n-2}{n-1}\xi_b\xi_c \sigma(R)h \\
 +\frac{1}{n-1}\abs{\xi}^2g_{bc}\sigma(R) \Big].
\end{multline}
We have
$$
R_{ab} = -\frac{1}{2} \Delta g_{ab} + \frac{1}{2} (\partial_a \Gamma_b + \partial_b \Gamma_a) +T_1
$$
and consequently, for a metric $g$ with $\abs{g}=1$, it holds that 
$$
R=\partial_a\Gamma^a +T_1.
$$
Substituting these formulas to~\eqref{Wc} yields the claim.
\end{proof}

\begin{proof}[Proof of Theorem \ref{thm_main_regularity}]
Let $g_{jk} \in C^r_*$, $r > 1$, in some system of local coordinates. By \cite[Theorem 2.1]{LiimatainenSalo} we know that $n$-harmonic coordinate systems exist near any point, and by \cite[Proposition 2.5]{LiimatainenSalo} the metric satisfies $g_{jk} \in C^r_*$ in $n$-harmonic coordinates.

Consider first the Weyl tensor. Let $n \geq 4$, $r > 1$, and $W_{abc}^{\phantom{abc}d} \in C^s_*$ for some $s > r-2$ in $n$-harmonic coordinates. Write $\hat{g} = \abs{g}^{-1/n} g_{jk}$ in these coordinates, so that 
$$
\hat{g}_{jk} \in C^r_*, \qquad \abs{\hat{g}} = 1.
$$
By conformal invariance we have $W_{abc}^{\phantom{abc}d}(\hat{g}) = W_{abc}^{\phantom{abc}d}(g)$  in $n$-harmonic coordinates where the right hand side is in $C^s_*$. This is a second order system for $\hat{g}$ where the principal symbol of the linearization is injective after applying the $n$-harmonic gauge condition \eqref{nharmonic_gauge_condition} and the condition $\abs{\hat{g}}=1$, by Lemma \ref{weyl_princip} and the algebraic computation in Lemma \ref{ellipt_calc} (note that ellipticity of $W_{abc}^{\phantom{abc}d}$ is equivalent with that of $W_{abcd}$). Hence Proposition \ref{elliptic_regularity_differential_operator} applies. If $s < r$ we obtain $\hat{g}_{jk} \in C^{s+2}_*$ as required. If $s \geq r$ we initially only get $\hat{g}_{jk} \in C^{\tilde{r}+2}_*$ for any $\tilde{r} < r$, but iterating the argument finitely many times yields $\hat{g}_{jk} \in C^{s+2}_*$ and we are done.

The proof for the Cotton and Fefferman-Graham obstruction tensors follows in a similar way from the above arguments by using the identities $(3-n) C_{abc} = \nabla^l W_{abcl}$ and $(n-3) \mathcal{O}_{ab}=\Delta^{n/2-2}\nabla^k\nabla^lW_{akbl}+T_{n-1}$.
\end{proof}

\begin{proof}[Proof of Theorem \ref{thm_main_conformalflatness}]
Assume first that $g_{jk} \in C^r_*$ where $n \geq 4$, $r > 1$, and the Weyl tensor vanishes near some point of $M$. By Theorem \ref{thm_main_regularity} we know that the metric $\hat{g}_{jk} = \abs{g}^{-1/n} g_{jk}$ is $C^{\infty}$ in any system of $n$-harmonic coordinates. Since the Weyl tensor of $\hat{g}$ vanishes, classical arguments (see for example~\cite[Chapter 4]{Aubin}) show that there is a possibly different set of coordinates $\{y^a\}$ (for which we use indices $a$, $b$) where $\hat{g}_{ab} = \hat{c} \delta_{ab}$ for some positive $C^{\infty}$ function $\hat{c}$. Then $g_{ab} = c \delta_{ab}$ where $c = \abs{g}^{1/n} \hat{c}$ is in $C^r_*$. The coordinates $\{y^{a}\}$ are again $n$-harmonic~\cite[Proposition 2.6]{LiimatainenSalo}. The argument for Cotton and Fefferman-Graham tensors is analogous.
\end{proof}

\appendix

\section{Elliptic regularity theorem} \label{app}

We give an elliptic regularity theorem for overdetermined elliptic systems. The result is essentially an integration of the techniques in~\cite[Chapter 3.2]{ToolsPDE} and Theorem 14.4.2 of~\cite{T3}. We adopt the notation of these references.

\begin{prop}\label{elliptic_regularity_differential_new}
Let $B \subset \mR^n$ be a ball, and let $p(x,D) = \sum_{\abs{\alpha} \leq m} p_{\alpha}(x) D^{\alpha}$ be an $M \times N$ matrix differential operator with coefficients in $C^r_*(B)$ where $r > 0$. Assume that $p(x,D)$ is overdetermined elliptic in the sense that its principal symbol $p_m(x,\xi)$ is injective for $x \in B$ and $\xi \neq 0$. If in the sense of distributions 
$$
p(x,D) u = f \quad \text{in } B
$$
where $u \in C^{m-r+\eps}_*(B)$ for some $\eps > 0$ and $f \in C^{\mu}_*(B)$ with $-r < \mu < r$, then $u \in C^{m+\mu}_*(B)$.
\end{prop}
\begin{proof}
We may assume that $B$ is centered at $0$. Let $B_0, B_1$ be balls centered at $0$ with $B_0 \subset \subset B_1 \subset \subset B$, let $\chi \in C^{\infty}_c(B_1)$ satisfy $\chi = 1$ near $\overline{B}_0$, and write $v = \chi u$. Writing $p_m(x,\xi) = \sum_{\abs{\alpha} = m} p_{\alpha}(x) \xi^{\alpha}$ we have 
 $$
 p_m(x,D)v = \chi f + \tilde{f}
 $$
 where $\tilde{f} \in C^{-r+1}_*(\mR^n)$ if $r > 1/2$ and $\tilde{f} \in C^r_*(\mR^n)$ if $0 < r \leq 1/2$.
 
We redefine the operator $p$ outside $\supp(\chi)$ as follows: let $K$ be the Kelvin transform for the ball $B_1$, and first define 
$$
q_{\alpha}(x) = \left\{ \begin{array}{ll} p_{\alpha}(x), & x \in \overline{B}_1, \\ p_{\alpha}(K(x)), & \text{otherwise}. \end{array} \right.
$$
Then $q_m(x,\xi) = \sum_{\abs{\alpha}=m} q_{\alpha}(x) \xi^{\alpha}$ is injective for all $x \in \mR^n$ and $\xi \neq 0$, and $q_m(x,\xi)$ is close to $p_m(0,\xi)$ if $\abs{x}$ is large. To obtain a symbol with $C^r_*$ regularity let $\tilde{\chi} \in C^{\infty}_c(B_1)$ satisfy $\tilde{\chi} = 1$ near $\supp(\chi)$,  and define 
$$
\tilde{p}_{\alpha} = \varphi_{\delta} \ast ((1-\tilde{\chi})q_{\alpha}) + \tilde{\chi} q_{\alpha}
$$
where $\varphi_{\delta}(x) = \delta^{-n} \varphi(x/\delta)$ and $\varphi \in C^{\infty}_c(\mR^n)$ is a standard mollifier supported in the unit ball with $\int_{\mR^n} \varphi(x) \,dx = 1$. One can check that $\tilde{p}_{\alpha} \in C^r_*(\mR^n)$, and if $\delta$ is small enough the symbol $\tilde{p}_m(x,\xi) = \sum_{\abs{\alpha}=m} \tilde{p}_{\alpha}(x) \xi^{\alpha}$ is injective for all $x$ and $\xi$, and moreover 
$$
\tilde{p}_m(x,\xi)^t \tilde{p}_m(x,\xi) \zeta \cdot \zeta \geq C  \abs{\zeta}^2 \abs{\xi}^{2m} , \quad x \in \mR^n, \ \ \xi \neq 0, \ \ \zeta \in \mC^n
$$
for some $C > 0$ by homogeneity and the fact that $\tilde{p}_m(x,\xi)$ is close to $p_m(0,\xi)$ for $\abs{x}$ large.
 
Since $\tilde{p}_{\alpha} = p_{\alpha}$ near $\supp(v)$, we have 
$$
\tilde{p}_m(x,D)v = \chi f + \tilde{f} \quad \text{in } \mR^n.
$$
Here $v \in C^{m-r+\eps}_*(\mR^n)$ and $\chi f \in C^{\mu}_*(\mR^n)$ with $-r < \mu < r$. If $0 < r \leq 1/2$ then $\tilde{f} \in C^{\mu}_*(\mR^n)$ and $u \in C^{m+\mu}_*(B_0)$ by using Proposition \ref{elliptic_regularity_new} below. If $r > 1/2$ then we have $\tilde{f} \in C^{-r+1}_*(\mR^n)$ instead, thus Proposition \ref{elliptic_regularity_new} implies $v \in C^{m+\min\{\mu,-r+1\}}_*$. If $\mu \leq -r+1$ we have $u \in C^{m+\mu}_*(B_0)$. Otherwise we observe that $u$ is $C^{m-r+1}_*$ near $\overline{B}_0$, and repeating the argument finitely many times with slightly smaller balls (and using the improved regularity of $\tilde{f}$) gives that $u \in C^{m+\mu}_*(B_0)$. Varying $B_0$ gives that $u$ is $C^{m+\mu}_*$ in $B$.
\end{proof}

\begin{prop}\label{elliptic_regularity_new}
Let $p \in C^r_* S^m_{1,0}(\mR^n)$ be an $M \times N$ matrix valued symbol where $m, r > 0$. Suppose that $p$ is overdetermined elliptic in the sense that there exist $C, K > 0$ such that  
\begin{equation}\label{ec}
p(x,\xi)^t p(x,\xi) \zeta \cdot \bar{\zeta} \geq C \abs{\zeta}^2|\xi|^{2m}, \quad x \in \mR^n, \ \ \abs{\xi} \geq K, \ \ \zeta \in \mC^n.
\end{equation}
If in the sense of distributions 
$$
p(x,D)u = f
$$
where $u \in C^{m-r+\eps}_{*}(\R^n)$ for some $\eps > 0$ and $f \in C^{\mu}_{*}(\R^n)$ with $-r < \mu < r$, then $u \in C^{m+\mu}_*(\R^n)$.
\end{prop}
\begin{proof}
We apply techniques from the calculus of pseudodifferential operators with nonsmooth coefficients. Let $\delta\in (0,1)$. Symbol smoothing allows us to write $p(x,\xi)$ as a sum of a smooth symbol of order $m$ and a symbol of order slightly less than $m$ with the same smoothness as $p(x,\xi)$, that is, 
$$
p(x,\xi) = p^{\sharp}(x,\xi) + p^{\flat}(x,\xi)
$$
with $p^{\sharp} \in S^m_{1,\delta}(\R^n)$ and $p^{\flat} \in C^{r}_* S^{m-\eps}_{1,\delta}(\R^n)$ with $\eps=r\delta>0$. See Lemma~\ref{symbol_smoothing} below. By the same lemma there is a left parametrix $E$ for $p^\sharp$ of class $OPS^{-m}_{1,\delta}(\R^n)$. The rest of the proof is exactly the same as the corresponding part in the proof of Theorem 14.4.2 in~\cite{T3}.
\end{proof}

\begin{lemma}\label{symbol_smoothing}
Let $p \in C^r_* S^m_{1,0}(\mR^n)$ be an $M \times N$ matrix valued symbol where $m, r > 0$. Let $\delta >0$. Then there is a symbol smoothing for $p(x,\xi)$ of the form
$$
p(x,\xi) = p^{\sharp}(x,\xi) + p^{\flat}(x,\xi)
$$
where $p^{\sharp} \in S^m_{1,\delta}(\R^n)$ and $p^{\flat} \in C_*^{r} S^{m-\eps}_{1,\delta}(\R^n)$ with $\eps=r\delta>0$.

If in addition $p(x,\xi)$ is overdetermined elliptic in the sense of~\eqref{ec}, then $p^\sharp$ is also overdetermined elliptic and there is a left parametrix for $p^\sharp$ that belongs to $OPS^{-m}_{1,\delta}(\R^n)$.
\end{lemma}

\begin{proof}
The decomposition is defined as in~\cite[Section 13.9]{T3}. That is, define
$$
p^{\sharp}(x,\xi) = \sum_{j=0}^{\infty} J_{\eps_j} p(x,\xi) \psi_j(\xi)
$$
where $(\psi_j)$ is a Littlewood-Paley partition of unity satisfying $\sum_{j=0}^{\infty} \psi_j = 1$, $\eps_j = 2^{-j\delta}$, and $J_\eps$ is the Fourier multiplier 
$$
J_{\eps} f(x) = (\phi(\eps D) f)(x)
$$
where $\phi \in C^{\infty}_c(\mR^n)$ with $\phi = 1$ for $\abs{\xi} \leq 1$. By exactly the same argument used in~\cite[Proposition 13.9.9]{T3}, we have that $p^{\sharp} \in S^m_{1,\delta}(\R^n)$ and $p^{\flat} \in C_*^{r} S^{m-\eps}_{1,\delta}(\R^n)$ with $\eps = r\delta$. 

By~\cite[Lemma 13.9.8]{T3}, $J_\eps$ satisfies
$$
||p(\cdot,\xi)-J_\eps p(\cdot,\xi)||_{L^\infty(\R^n)}\leq C_{r}\eps^{r} ||p(\cdot,\xi)||_{C^r_*(\R^n)},
$$
with $C_{r}$ independent of $\eps$. We also have the estimates
\begin{align*}
 ||p(x,\xi)||&\leq C_1 |\xi|^m \\
 ||p^\sharp(x,\xi)||&\leq C_2 |\xi|^m \\
 ||p(\cdot,\xi)||_{C_*^r(\R^n)}&\leq C_3 |\xi|^m,  
\end{align*}
holding for $x \in \mR^n$ and $|\xi|\geq K_2>0$ since $p(x,\xi)\in C^r_*S^m_{1,0}(\R^n)$ and $p^\sharp(x,\xi)\in S^m_{1,\delta}(\R^n)$; see~\cite[p. 46]{T3}. The norms are defined with respect to the Hilbert-Schmidt matrix norm. Choose $M$ so large that
$$
C_{r} \eps_j^{r}\leq \frac{1}{2C_3(C_1 + C_2)}C
$$
for all $j\geq M$. Let us define $L=\max(K,K_2,2^{M-1})$.

Let $\xi\in \R^n$ with $|\xi|\geq L$. It follows from the definition of the Littlewood-Paley partition $(\psi_j)$ that $\psi_j(\xi)\equiv0$ if $j < M$. Thus we have for $|\xi|\geq L$ the estimate
\begin{align*}
||p(\cdot,\xi)-p^\sharp(\cdot,\xi)||&_{L^\infty(\R^n)}\leq \sum_{j \geq M}||p(\cdot,\xi)-J_{\eps_j}p(\cdot,\xi)||_{L^\infty(\R^n)}\psi_j(\xi) \\
  &\leq C_r \sum_{j\geq M} \eps_j^r||p(\cdot,\xi)||_{C_*^{r}(\R^n)}\psi_j(\xi)\leq \frac{1}{2(C_1 + C_2)}C \abs{\xi}^m.
\end{align*}
Subsequently, we have
\begin{multline*}
||p(x,\xi)^tp(x,\xi)-p^\sharp(x,\xi)^tp^\sharp(x,\xi)||   \leq ||p(x,\xi)^t||
||p(x,\xi) -p^\sharp(x,\xi)|| \\ +||p(x,\xi)^t-p^\sharp(x,\xi)^t||||p^\sharp(x,\xi)||  \leq \frac{1}{2}C \abs{\xi}^{2m}
\end{multline*}
whenever $\abs{\xi}\geq L$. It follows that 
\begin{align}\label{distance_between}
p^\sharp&(x,\xi)^t p^\sharp(x,\xi) \zeta \cdot \bar{\zeta} = \l(p^\sharp(x,\xi)^t p^\sharp(x,\xi)- p(x,\xi)^t p(x,\xi)\r)\zeta \cdot \bar{\zeta}\nonumber \\
  &+ p(x,\xi)^t p(x,\xi) \zeta \cdot \bar{\zeta} \geq - \abs{\zeta}^2||p(x,\xi)^tp(x,\xi)-p^\sharp(x,\xi)^tp^\sharp(x,\xi)||\nonumber  \\ 
  &+ C\abs{\zeta}^2\abs{\xi}^{2m}\geq \frac{1}{2} C \abs{\zeta}^2\abs{\xi}^{2m},
\end{align}
for $x \in \mR^n$,  $\abs{\xi} \geq L$, $\zeta \in \mC^n$. Thus $p^\sharp$ is overdetermined elliptic.

From~\eqref{distance_between} it follows that the smallest eigenvalue of $p^\sharp(x,\xi)^t p^\sharp(x,\xi)$ is at least $C\abs{\xi}^{2m}/2$ for $\abs{\xi}\geq L$.
Therefore, the inverse of $p^\sharp(x,\xi)^tp^\sharp(x,\xi)$ satisfies
$$
||\l(p^\sharp(x,\xi)^tp^\sharp(x,\xi)\r)^{-1}||\leq \frac{N^{1/2}}{\frac{1}{2}\abs{\xi}^{2m}}
$$
for $\abs{\xi}\geq L$, showing that $p^\sharp(x,\xi)^tp^\sharp(x,\xi)$ is an elliptic $N \times N$ matrix symbol in $S^{2m}_{1,\delta}(\R^n)$. Thus $(p^{\sharp})^t(x,D) p^\sharp(x,D)$ has a parametrix, say $Q$, of class $OPS^{-2m}_{1,\delta}(\R^n)$; see e.g.~\cite[Ch. 4]{T2}. Now, 
$$
E=Q\,(p^{\sharp})^t(x,D)
$$
is a left parametrix for $p^\sharp(x,D)$ of class $OPS^{-m}_{1,\delta}(\R^n)$.
\end{proof}

\end{document}